\documentclass{amsart}
\usepackage{graphicx}
\newtheorem{thm}{Theorem}[section]

\newtheorem{lem}[thm]{Lemma}
\newtheorem{prop}[thm]{Proposition}
\theoremstyle{definition}

\theoremstyle{remark}
\newtheorem{rem}[thm]{Remark}
\numberwithin{equation}{section}

\begin{document}

\title[Bounds for the loss probability]
{Bounds for the loss probability in large loss queueing systems
}%
\author{Vyacheslav M. Abramov}%
\address{School of Mathematical Sciences, Monash University, Clayton
Campus,
Wellington road, Victoria-3800, Australia}%
\email{vyacheslav.abramov@sci.monash.edu.au}%


\begin{abstract}
Let $\mathcal{G}(\frak{g}_1,\frak{g}_2)$ be the class of all
probability distribution functions of positive random variables
having the given first two moments $\frak{g}_1$ and $\frak{g}_2$. Let
$G_1(x)$ and $G_2(x)$ be two probability distribution functions of
this class satisfying the condition $|G_1(x)-G_2(x)|<\epsilon$ for
some small positive value $\epsilon$ and let $\widehat{G}_1(s)$ and,
respectively, $\widehat{G}_2(s)$ denote their Laplace-Stieltjes
transforms. For real $\mu$ satisfying $\mu\frak{g}_1>1$ let
us denote by $\gamma_{G_1}$ and $\gamma_{G_2}$ the least positive
roots of the equations $z=\widehat{G}_1(\mu-\mu z)$ and
$z=\widehat{G}_2(\mu-\mu z)$ respectively. In the paper, the upper
bound for $|\gamma_{G_1}-\gamma_{G_2}|$ is derived. This upper bound
is then used to find lower and upper bounds for the loss probabilities
in different large loss queueing systems.
\end{abstract}
\maketitle
\section{Introduction}\label{sec1}
In most of stochastic models studied analytically in the literature the
probability distribution functions of their random characteristics are
assumed to be known. In queueing problems, for example, the input
characteristics are the distributions of interarrival and service
times, and they are clearly described in the formulation of a
problem. For example in the case of an $M/G/1$ queueing system,
the arrival process is usually assumed to be Poisson of rate
$\lambda$, and service time distribution is assumed to be a given function
$B(x)$, with mean $1/\mu$ and other moments if required. This
enables us to use the techniques of the Laplace-Stieltjes
transforms or generating functions to obtain the desired output
characteristics.

In practice, however, the distribution of an interarrival or service
time is unknown. It can be only \textit{approximated} by available information about
that distribution, and the accuracy of that approximation can be
obtained by analysis of real observations.

The problems of modeling, approximating and estimating the output
characteristics of queueing systems are a well-known, and it is a
well-established and distinguished area of queueing theory. There is
a wide literature related to this subject. To mention only a few
papers that use different approaches, we refer Bareche and Aissani
\cite{Bareche and Aissani 2008}, Kalashnikov \cite{Kalashnikov 2002}
and van Dijk and Miyazawa \cite{van Dijk and Miyazawa 2004}. Bareche
and Aissani \cite{Bareche and Aissani 2008} used the strong
stability method to study the error of approximation of $GI/M/1$ or
$M/GI/1$ queueing systems by that $M/M/1$, when the distribution of
inter-arrival time or, respectively, service time is unknown but in
the certain sense (that defined in that paper) close to the
exponential distribution. Kalashnikov \cite{Kalashnikov 2002}
studied stochastic sequences satisfying the recurrence relation
$V_{n+1}=F(V_n,\xi_n)$, where $\xi_n$ was a sequence of independent
and identically distributed finite-dimensional random vectors. By
replacing the original sequence $\{\xi_n\}$ by ``perturbed" sequence
$\{\xi_n^\prime\}$, under the assumption that a specially defined
weighted distance between $\xi_n$ and $\xi_n^\prime$ is given, that
weighted distance between $V_n$ and $V_{n}^\prime$, where
$V_{n+1}^\prime=F(V_n^\prime,\xi_n^\prime)$, has been studied. Van
Dijk and Miyazawa \cite{van Dijk and Miyazawa 2004} studied
non-exponential queues such as $GI/GI/1/n$, $M/GI/c/n$ and
$GI/M/c/n$.  For the $GI/GI/1/n$ queueing system they demonstrated the
influence of an error in the service time distribution on the resulting error in
different performance measures such as the throughout of the system. They
also established the error bounds for the throughout of the
$M/GI/c/n$ queue, and obtained similar results for the $GI/M/c/n$ queue
with a perturbation of the interarrival time distribution.

In the present paper, we study the class $\mathcal{G}(\frak{g}_1,\frak{g}_2)$ of probability distribution functions of positive random variables having the given first two moments $\frak{g}_1$ and $\frak{g}_2$.
We establish the bounds for the least
positive root of the functional equation $z=\widehat{G}(\mu-\mu
z)$, where $\widehat{G}(s)$ is the Laplace-Stieltjes transform of
an unknown probability distribution function $G(x)$ belonging to the class $\mathcal{G}(\frak{g}_1,\frak{g}_2)$, and $\mu$ is a positive parameter satisfying the
condition $\mu\frak{g}_1>1$. The additional information
characterizing $G(x)$ is that
\begin{equation}\label{1.2}
\mathcal{K}(G,F):=\sup_{x>0}|G(x)-F(x)|<\epsilon,
\end{equation}
where $F(x)$ is known probability distribution function of a positive random variable having the same two moments (i.e. belonging to the class $\mathcal{G}(\frak{g}_1,\frak{g}_2)$ as well), and the least positive root, $\gamma_F$, of the functional equation $z=\widehat{F}(\mu-\mu z)$ is therefore known ($\widehat{F}(s)$ denotes the Laplace-Stieltjes transform of $F(x)$). The metric $\mathcal{K}(G, F)$ is known as the uniform
(Kolmogorov's) metric (e.g. \cite{Kalashnikov and Rachev 1990},
\cite{Rachev 1991}).

The aforementioned bounds for the least positive root $\gamma_G$ of the
functional equation $z=\widehat{G}(\mu-\mu z)$ (or similar functional equations)
are then used in
asymptotic analysis of the loss probability in certain queueing
systems with the large number of waiting places.

There are two areas of applications where these bounds are used. They are
 \textit{statistics} of queueing
systems and \textit{continuity} of queueing systems.

In statistical problems, the empirical
probability distribution ${G}_{\mathrm{emp}}(x,N)$ ($N$ is the number of observations)
 is assumed to be known. If the
number of observations increases to infinity, then for any given
positive value $\epsilon$ the probability
$\mathrm{P}\left\{\sup_{x>0}\Big|{G}_{\mathrm{emp}}(x,N)-G(x)\Big|<\epsilon\right\}$
approaches 1.

More exact information about this probability is given by Kolmogorov's theorem (see Kolmogoroff \cite{Kolmogorov} or Tak\'acs \cite{Takacs 1967}, p.170). Namely,
\begin{equation*}
\lim_{N\to\infty}\mathrm{P}\left\{\sup_{x\geq0}\big|{G}_{\mathrm{emp}}(x,N)-G(x)\big|\leq
\frac{z}{\sqrt{N}}\right\}=K(z),
\end{equation*}
where
$$
K(z)=\begin{cases}\sum\limits_{j=-\infty}^{+\infty}(-1)^j\mathrm{e}^{-2j^{2}z^{2}}, &\mbox{for} \ z>0,\\
0, &\mbox{for} \ z\leq0.
\end{cases}
$$
So, the probability of
\begin{equation}\label{sp}\sup_{x\geq0}|G(x)-{G}_{\mathrm{emp}}(x,N)|<\epsilon,\end{equation}
can be asymptotically evaluated when $N$ is large and $\epsilon$ is small. (For relevant studies associated with statistics \eqref{sp} or other related statistics see the book of Tak\'acs \cite{Takacs 1967}.)

The first and second moments of $G(x)$ are usually unknown either.
However, for large $N$, they can be taken approximately to the
empirical moments of ${G}_{\mathrm{emp}}(x,N)$ with some error. It
will be shown below (see Theorem \ref{thm1}, rel. \eqref{T1} and
Remark \ref{rem1}) that if in relation \eqref{sp} the value $\epsilon$ is small enough,
then the bounds for the least positive root $\gamma_G$ are expressed via the
first moment $\frak{g}_1$ only. In this case the only error of the
empirical mean is to be taken into account. Thus, in the motivation of
assumption \eqref{sp}, the value $\epsilon$ is assumed to be chosen
such small that the probability of \eqref{sp} should be large on the one hand,
and the bounds for the empirical mean
should be small on the other hand.

In continuity problems, we assume that the unknown probability
distribution $G(x):=\mathrm{P}\{\zeta\leq x\}$ with the expectation
$\frak{g}_1:=\frac{1}{\lambda}$ satisfies some specific properties
such as
$$
\sup_{x>0, y>0}\big|G(x)-\mathrm{P}\{\zeta\leq
x+y|\zeta>y\}\big|<\epsilon.
$$
Then, according to the known characterization theorem of Azlarov
and Volodin (see \cite{Azlarov and Volodin 1986} or \cite{Abramov
2008}), we have
$$
\sup_{x>0}\big|G(x)-(1-\mathrm{e}^{-\lambda x})\big|<2\epsilon.
$$
For other related continuity problems see \cite{Abramov 2008}, where Kolmogorov's metric is used for
continuity analysis of the $M/M/1/n$ queueing system.

The class of probability distributions functions
$\mathcal{G}(\frak{g}_1, \frak{g}_2)$ itself, i.e. without
metrical condition \eqref{1.2}, has been studied by Vasilyev
and Kozlov \cite{Vasilyev and Kozlov 1969} and Rolski \cite{Rolski
1972}. Rolski \cite{Rolski 1972} has established the bounds for
the least positive root of the functional equation
$z=\widehat{G}(\mu-\mu z)$.

In the present paper, we show that additional condition
\eqref{1.2} nontrivially improves the earlier bounds obtained by
Rolski \cite{Rolski 1972}. The new bounds have various applications. For example, the upper and lower
asymptotic bounds can be obtain for the loss probabilities in $M/GI/1/n$, $GI/M/1/n$
and $GI/M/m/n$ queueing systems with large capacity $n$ as well as
in many related models of telecommunication systems (see \cite{Abramov
1997}, \cite{Abramov 2002}, \cite{Abramov 2004}, \cite{Abramov
2007}, \cite{Abramov 2008b} and \cite{Abramov 2008c}).
We demonstrate application of this theory for the $GI/M/1/n$
queueing system with large buffer capacity $n$, for the $M/GI/1$ buffer system with two types of losses \cite{Abramov 2004} and then for the special
buffers model with batch service and priorities \cite{Abramov 2008b}. The last two of the mentioned applications have especial importance for telecommunication systems.
We also establish new continuity results for the loss probability
in the $M/M/1/n$ queueing systems with large capacity $n$ under
special assumptions related to interarrival times. The
 continuity theorems for $M/M/1/n$
queueing systems, where the buffer capacity $n$ is fixed, have been established in
\cite{Abramov 2008}. Statistical analysis of $M/GI/1/n$ and $GI/M/1/n$ loss systems with fixed buffer capacity $n$ based on Kolmogorov's statistics has been provided in \cite{Abramov 2009}.

The paper is structured as follows. In Section 2, properties
of distributions belonging to the class $\mathcal{G}(\frak{g}_1,
\frak{g}_2)$ and satisfying additional condition \eqref{1.2} are
studied. Let $G_1(x)$ and $G_2(x)$ be arbitrary probability
distribution functions of this class satisfying \eqref{1.2}, i.e. $G_1, G_2\in\mathcal{G}(\frak{g}_1,\frak{g}_2)$, and $\mathcal{K}(G_1,G_2)<\epsilon$.
Denote by $\widehat{G}_1(s)$ and, respectively, by
$\widehat{G}_2(s)$ ($s\geq0$) their Laplace-Stieltjes transforms.
Let $\gamma_{G_1}$ and $\gamma_{G_2}$
 be the corresponding
solutions of the functional equations $z=\widehat{G}_1(\mu-\mu z)$
and $z=\widehat{G}_2(\mu-\mu z)$ both belonging to the interval
(0,1). (Recall that according to the well-known theorem of
Tak\'acs \cite{Takacs 1955}, under the assumption
$\mu\frak{g}_1>1$ the least positive roots $\gamma_{G_1}$ and $\gamma_{G_2}$ of the
equations $z=\widehat{G}_1(\mu-\mu z)$ and
$z=\widehat{G}_2(\mu-\mu z)$ are unique in the interval (0,1).)  An
upper bound for $|\gamma_{G_1}-\gamma_{G_2}|$ is obtained in Section
\ref{sec2}. In Sections 3 and 4, applications of the results of Section \ref{sec2} are given for different loss queueing systems. Specifically, in Section \ref{sec3.1} lower and upper asymptotic bounds are established for loss probabilities in the $GI/M/1/n$ queueing system as $n$ increases to infinity; in Section \ref{sec3.1A}, bounds for the loss probability in $M/GI/1$ buffer model with two types of losses, which has been studied in \cite{Abramov 2004}, are obtained, and
in
Section \ref{sec3.2}, bounds for the loss probabilities in the
buffers model with priorities, which has been studied in \cite{Abramov 2008b}, are established. In Section
\ref{sec4}, the continuity analysis of the loss probability in the
$M/M/1/n$ queueing system is provided. The continuity analysis of Section \ref{sec4} is
based on the bounds obtained in Section \ref{sec2}, the results
for the loss probabilities obtained in Section \ref{sec3.1} and
characterization properties of the exponential distribution.

\section{Properties of probability distribution functions of the class
$\mathcal{G}$}\label{sec2} In this section we establish an
inequality for $|\gamma_{G_1}-\gamma_{G_2}|$ for probability distribution functions $G_1(x)$ and $G_2(x)$ belonging to the class $\mathcal{G}(\frak{g}_1, \frak{g}_2)$ and satisfying the condition
\begin{equation}\label{AC}
\sup_{x>0}|G_1(x)-G_2(x)|<\epsilon.
\end{equation}

 We start from the known
inequalities for probability distribution functions of the class
$\mathcal{G}(\frak{g}_1, \frak{g}_2)$. Vasilyev and Kozlov
\cite{Vasilyev and Kozlov 1969} proved,
\begin{equation}\label{2.1}
\inf_{G\in\mathcal{G}(\frak{g}_1,
\frak{g}_2)}\int_0^\infty\mathrm{e}^{-sx}\mathrm{d}G(x)=\mathrm{e}^{-s\frak{g}_1},
\ s\geq0
\end{equation}
and
\begin{equation}\label{2.2}
\max_{G\in\mathcal{G}(\frak{g}_1,
\frak{g}_2)}\int_0^\infty\mathrm{e}^{-sx}\mathrm{d}G(x)=1-\frac{\frak{g}_1^2}{\frak{g}_2}+
\frac{\frak{g}_1^2}{\frak{g}_2}\exp\left(-\frac{\frak{g}_2}{\frak{g_1}}s\right),
\ s\geq0,
\end{equation}
where the maximum is obtained for
\begin{equation}\label{2.3}
G(x)=G_{\mathrm{max}}(x)=\begin{cases}0, &\text{if} \ t<0;\\
1-\frac{\frak{g}_1^2}{\frak{g}_2}, &\text{if} \ 0\leq
t<\frac{\frak{g}_2}{\frak{g}_1};\\
1, &\text{if} \ t\geq\frac{\frak{g}_2}{\frak{g}_1}.
\end{cases}
\end{equation}

The lower and upper bounds given by \eqref{2.1} and \eqref{2.2}
are tight. If $\frak{g}_2=\frak{g}_1^2$, then these bounds
coincide.

It is pointed out in Rolski \cite{Rolski 1972} that \eqref{2.1} and
\eqref{2.2} could be obtained immediately by the method of
reduction to the Tchebycheff system \cite{Karlin and Studden 1966}
if one takes into account that $\{1, t, t^2\}$ and $\{1, t, t^2,
\mathrm{e}^{-st}\}$ form Tchebycheff systems on $[0, \infty)$.
Rolski \cite{Rolski 1972} has established as follows. For
$\gamma_G$, the least positive root of the functional
equation $z=\widehat{G}(\mu-\mu z)$, it was shown
\begin{equation}\label{2.4}
\inf_{G\in\mathcal{G}(\frak{g}_1, \frak{g}_2)}\gamma_G=\ell,
\end{equation}
and
\begin{equation}\label{2.5}
\max_{G\in\mathcal{G}(\frak{g}_1,
\frak{g}_2)}\gamma_G=\gamma_{G_{\mathrm{max}}}=1+\frac{\frak{g}_1^2}{\frak{g}_2}(\ell-1),
\end{equation}
where $\ell$ in \eqref{2.4} and \eqref{2.5} is the least root of
the equation:
\begin{equation}\label{LR}
z=\mathrm{e}^{-\mu\frak{g}_1+\mu\frak{g}_1z}.
\end{equation}

The proof of \eqref{2.4} and \eqref{2.5} given in \cite{Rolski
1972} is based on the convexity of the function
$\widehat{G}(\mu-\mu z)-z$.

From \eqref{2.1} and \eqref{2.2} we also have as
follows. Let $G_1(x)$ and $G_2(x)$ be arbitrary probability
distribution functions of the class $\mathcal{G}(\frak{g}_1,
\frak{g}_2)$, and let $\widehat{G}_1(s)$ and, correspondingly,
$\widehat{G}_2(s)$ be their Laplace-Stieltjes transforms
($s\geq0$). Then,
\begin{equation}\label{2.6}
\sup_{G_1,G_2\in\mathcal{G}(\frak{g}_1,
\frak{g}_2)}\sup_{s\geq0}\big|\widehat{G}_1(s)-\widehat{G}_2(s)\big|=1-\frac{\frak{g}_1^2}{\frak{g}_2}.
\end{equation}
Indeed, for the derivative of the difference between the
right-hand side of \eqref{2.2} and that of \eqref{2.1} we have
\begin{equation}\label{2.7}
\begin{aligned}
&\frac{\mathrm{d}}{\mathrm{d}s}\left[1-\frac{\frak{g}_1^2}{\frak{g}_2}+
\frac{\frak{g}_1^2}{\frak{g}_2}\exp\left(-\frac{\frak{g}_2}{\frak{g_1}}s\right)-\mathrm{e}^{-s\frak{g}_1}\right]\\
&=\frak{g}_1\left(\exp(-\frak{g}_1s)-\exp\Big(-\frac{\frak{g}_2}{\frak{g}_1}s\Big)\right).
\end{aligned}
\end{equation}
This derivative is equal to zero for $s=0$ (minimum) and
$s=+\infty$ (maximum). (The trivial case
$\frak{g}_2=\frak{g}_1^2$, leading to the identity to zero of the
right-hand side of \eqref{2.7} for all $s\geq0$, is not
considered.)

Therefore, from \eqref{2.7} as well as from \eqref{2.1} and
\eqref{2.2} we arrive at \eqref{2.6}.

In turn, from \eqref{2.4} and \eqref{2.5} we have the following
inequality for $|\gamma_{G_1}-\gamma_{G_2}|$:
\begin{equation}\label{2.8}
|\gamma_{G_1}-\gamma_{G_2}|\leq1+\frac{\frak{g}_1^2}{\frak{g}_2}(\ell-1)-\ell.
\end{equation}

The inequality
\eqref{2.8} follows from the results of Rolski \cite{Rolski 1972}. Under additional condition \eqref{AC} we will establish
an improved inequality for $|\gamma_{G_1}-\gamma_{G_2}|$.

Prior studying the properties of the class of probability distribution functions $\mathcal{G}(\frak{g}_1, \frak{g}_2)$ under additional condition \eqref{AC}, note that
inequalities \eqref{2.1}, \eqref{2.2},
\eqref{2.4}, and \eqref{2.5} hold
true for a wider class of probability distribution functions
than $\mathcal{G}(\frak{g}_1, \frak{g}_2)$. We will prove that the above inequalities remain correct for the class of probability distribution functions $\bigcup_{(m_1,
m_2)\in\mathcal{M}(\frak{g}_1, \frak{g}_2)} \mathcal{G}(m_1, m_2)$, where the set of pairs $\{(m_1,m_2)\}$ contains the pair $(\frak{g}_1, \frak{g}_2)$ (this set of pairs denoted by $\mathcal{M}(\frak{g}_1, \frak{g}_2)$ will be defined below).

\smallskip

Let
$\frak{m}>\frac{1}{\mu}$ be such the boundary value, that the
least root of the equation
$$z=\mathrm{e}^{-\mu\frak{m}+\mu\frak{m}z}$$ is equal to the
right-hand side of \eqref{2.5}, and let $m_1$ and $m_2$ are the values
satisfying the inequalities $\frak{m}\leq m_1\leq\frak{g}_1$, and
$\frac{m_1^2}{m_2}\geq\frac{\frak{g}_1^2}{\frak{g}_2}$ ($m_1^2\leq
m_2$). Then, we have the same bounds \eqref{2.1} and \eqref{2.2} for the
probability distribution functions and \eqref{2.4} and \eqref{2.5}
for the roots $\gamma_G$ but now for the wider class of
probability distribution functions belonging to
$\mathcal{G}(\frak{g}_1, \frak{g}_2)\cup \mathcal{G}(m_1, m_2)$.

Indeed, for any $m_1$ satisfying the inequality $\frak{m}\leq
m_1\leq\frak{g}_1$, and any $m_2$ for which
$\frac{m_1^2}{m_2}\geq\frac{\frak{g}_1^2}{\frak{g}_2}$, according
to \cite{Vasilyev and Kozlov 1969} we have
\begin{equation}\label{2.1'}
\inf_{\mathcal{G}(m_1,
m_2)}\int_0^\infty\mathrm{e}^{-sx}\mathrm{d}G(x)=\mathrm{e}^{-sm_1}\geq
\mathrm{e}^{-s\frak{g}_1}, \ s\geq0,
\end{equation}
and, taking into account that
$\frac{m_1^2}{m_2}\geq\frac{\frak{g}_1^2}{\frak{g}_2}$ and
$m_1\leq\frak{g}_1$ together lead to
$\frac{m_1}{m_2}\geq\frac{\frak{g}_1}{\frak{g}_2}$, we also have
\begin{equation}\label{2.2'}
\begin{aligned}
\max_{G\in\mathcal{G}(m_1,
m_2)}\int_0^\infty\mathrm{e}^{-sx}\mathrm{d}G(x)&=1-\frac{m_1^2}{m_2}+
\frac{m_1^2}{m_2}\exp\left(-\frac{m_2}{m_1}s\right)\\
&\leq1-\frac{\frak{g}_1^2}{\frak{g}_2}+
\frac{\frak{g}_1^2}{\frak{g}_2}\exp\left(-\frac{\frak{g}_2}{\frak{g_1}}s\right),
\ s\geq0,
\end{aligned}
\end{equation}
where the equality in the right-hand side of the first line of \eqref{2.2'} is a replacement of the initial probability distribution function (given by \eqref{2.3}) by
another one, where the parameters $\frak{g}_1$ and $\frak{g}_2$ are correspondingly replaced with $m_1$ and $m_2$.

According to the result of Rolski \cite{Rolski 1972}, we respectively have:
\begin{equation}\label{2.4'}
\inf_{G\in\mathcal{G}(m_1, m_2)}\gamma_G=\ell^*\geq\ell,
\end{equation}
and
\begin{equation}\label{2.5'}
\max_{G\in\mathcal{G}(m_1,
m_2)}\gamma_G=1+\frac{m_1^2}{m_2}(\ell^*-1)\leq
1+\frac{\frak{g}_1^2}{\frak{g}_2}(\ell-1).
\end{equation}

From the above inequalities of \eqref{2.1'} - \eqref{2.5'}, one can
conclude as follows. Let
$$\mathcal{M}(\frak{g}_1, \frak{g}_2) =
\left\{(m_1, m_2): \frak{m}\leq m_1\leq\frak{g}_1;\quad
\frac{m_1^2}{m_2}\geq\frac{\frak{g}_1^2}{\frak{g}_2};\quad
m_1^2\leq m_2\right\}.
$$
(Recall that $\frak{m}>\frac{1}{\mu}$ is such the boundary value
that the least root of the equation
$z=\mathrm{e}^{-\mu\frak{m}+\mu\frak{m}z}$ is equal to the
right-hand side of \eqref{2.5}.) Denote
$$
\mathcal{G}(\mathcal{M})=\bigcup_{(m_1,
m_2)\in\mathcal{M}(\frak{g}_1, \frak{g}_2)} \mathcal{G}(m_1, m_2).
$$

 Then we have the following elementary generalization of
\eqref{2.4} and \eqref{2.5}:
\begin{equation}\label{2.4''}
\inf_{G\in\mathcal{G}(\frak{g}_1, \frak{g}_2)}\gamma_G=
\inf_{G\in\mathcal{G}(\mathcal{M})}\gamma_G=\ell,
\end{equation}
and
\begin{equation}\label{2.5''}
\max_{G\in\mathcal{G}(\frak{g}_1,
\frak{g}_2)}\gamma_G=\max_{G\in\mathcal{G}(\mathcal{M})}
\gamma_G=\gamma_{G_{\mathrm{max}}}=1+\frac{\frak{g}_1^2}{\frak{g}_2}(\ell-1).
\end{equation}
Notice that if $m_1=\frak{m}$, then we have
$\ell^*=1+\frac{\frak{g}_1^2}{\frak{g}_2}(\ell-1)$, where $\ell^*$
is defined by \eqref{2.4'}. On the other hand, according to
\eqref{2.5'} we obtain $\frac{m_1^2}{m_2}=1$, i.e. in this case
$m_2=\frak{m}^2$. Thus the set $\mathcal{M}(\frak{g}_1,
\frak{g}_2)$ and, consequently, the class
$\mathcal{G}(\mathcal{M})$ are defined correctly.

We start now to work with \eqref{AC}. We have the following
elementary property:
\begin{equation}\label{2.9}
\begin{aligned}
\sup_{s\geq0}\big|\widehat{G}_1(s)-\widehat{G}_2(s)\big|&=\sup_{s>0}\left|\int_0^\infty\mathrm{e}^{-sx}
\mathrm{d}G_1(x)-\int_0^\infty\mathrm{e}^{-sx}
\mathrm{d}G_2(x)\right|\\
&\leq\sup_{s>0}\int_0^\infty
s\mathrm{e}^{-sx}\underbrace{\sup_{y\geq0}\big|G_1(y)-G_2(y)\big|}_{\leq\epsilon\text{ according to }\eqref{AC}}\mathrm{d}x\\
&\leq\epsilon.
\end{aligned}
\end{equation}
Thus under the assumption of \eqref{AC}, the difference in
absolute value between the Laplace-Stieltjes transforms
$\widehat{G}_1(s)$ and $\widehat{G}_2(s)$ is not greater than
$\epsilon$.

It follows from \eqref{2.9} that
\begin{equation}\label{2.9'}
\sup_{\substack{G_1,G_2\in\mathcal{G}(\frak{g}_1, \frak{g}_2)\\
\mathcal{K}(G_1,G_2)\leq\epsilon}}\sup_{s\geq0}\big|\widehat{G}_1(s)-\widehat{G}_2(s)\big|=\epsilon_1\leq\epsilon.
\end{equation}
(We do not know whether or not the value $\epsilon_1$ can be found.
However, the exact value of $\epsilon_1$ is not important for our
further considerations. Relation \eqref{2.9'} will be used later
in this section.)

On the other hand, according to \eqref{2.6} for two arbitrary
probability distribution functions of the class
$\mathcal{G}(\mathcal{M})$ the difference in absolute value
between their Laplace-Stieltjes transforms is not greater than
$1-\frac{\frak{g}_1^2}{\frak{g}_2}$. Therefore, if
$\epsilon\geq1-\frac{\frak{g}_1^2}{\frak{g}_2}$, then the condition
\eqref{1.2} is not meaningful. Therefore, it will be assumed in
the further consideration that
$\epsilon<1-\frac{\frak{g}_1^2}{\frak{g}_2}$.

\smallskip

The lemma below is the statement on the dense of the class
$\mathcal{G}(\frak{g}_1, \frak{g}_2)$.

\begin{lem}\label{lem1}
For any probability distribution function
$G(x)\in\mathcal{G}(\frak{g}_1, \frak{g}_2)$ \
($\frak{g}_1^2\neq\frak{g}_2$) there exists another probability
distribution function $\widetilde{G}(x)\in\mathcal{G}(\frak{g}_1,
\frak{g}_2)$, which distinguishes from $G(x)$ at least in one point, such that for any $\delta>0$,
\begin{equation*}
\mathcal{K}\big(\widetilde{G}, G\big)<\delta.
\end{equation*}
\end{lem}

\begin{proof} Under the assumption that the class $\mathcal{G}(\frak{g}_1, \frak{g}_2)$ is not trivial, i.e. $\frak{g}_2\neq\frak{g}_1^2$, one can take two distinct probability distribution functions $G(x)$ and $F(x)$ of this class. For any $p\in(0,1)$, let $G_p(x)=pF(x)+(1-p)G(x)$. Apparently, $G_p(x)\in\mathcal{G}(\frak{g}_1, \frak{g}_2)$ as well. Therefore, choosing $p<\frac{\delta}{2}$, by the triangle inequality we obtain:
$$
|G_p(x)-G(x)|=|pF(x)+(1-p)G(x)-G(x)|\leq pF(x)+pG(x)<\frac{\delta}{2}+\frac{\delta}{2}=\delta.
$$
Hence, for $p<\frac{\delta}{2}$ one can set $\widetilde{G}(x)=G_p(x)$.
\end{proof}

An extended version of Lemma \ref{lem1} is given in the following lemma.

\begin{lem}\label{cor1}Let $(m_1, m_2)\in\mathcal{M}$ and $(m_1^\prime, m_2^\prime)\in\mathcal{M}$ ($m_2\neq m_1^2$, $m_2^\prime\neq (m_1^\prime)^2$), and let $G(x)\in\mathcal{G}(m_1, m_2)$. Then for any $\delta>0$ there exists a probability distribution function $\widetilde{G}(x)\in\mathcal{G}(m_1^\prime, m_2^\prime)$ such that $\sup_{x>0}|G(x)-\widetilde{G}(x)|<\delta$.
\end{lem}

\begin{proof} Assume that $m_2^\prime\geq m_2$. Then take a probability distribution function $F(x)$ satisfying the properties:
\begin{equation}\label{COR2}
\int_{-\infty}^\infty x\mathrm{d}F(x)=\frac{2m_1^\prime-(2-\delta)m_1}{\delta},
\end{equation}
and
\begin{equation}\label{COR3}
\int_{-\infty}^\infty x^{2}\mathrm{d}F(x)=\frac{2m_2^\prime-(2-\delta)m_2}{\delta}.
\end{equation}
Then, the probability distribution function
\begin{equation}\label{COR1}
\widetilde{G}(x)=\left(1-\frac{\delta}{2}\right)G(x)+\frac{\delta}{2}F(x)
\end{equation}
belongs to the class $\mathcal{G}(m_1^\prime, m_2^\prime)$, and, according to the triangle inequality
$$
|\widetilde{G}(x)-G(x)|\leq\frac{\delta}{2}G(x)+\frac{\delta}{2}F(x)<\delta.
$$
In the opposite case where $m_2^\prime< m_2$ take the probability distribution function $F(x)$ satisfying the properties
\begin{equation*}\label{COR4}
\int_{-\infty}^\infty x\mathrm{d}F(x)=\frac{2m_1-(2-\delta)m_1^\prime}{\delta},
\end{equation*}
and
\begin{equation*}\label{COR5}
\int_{-\infty}^\infty x^{2}\mathrm{d}F(x)=\frac{2m_2-(2-\delta)m_2^\prime}{\delta}.
\end{equation*}
Then, instead of \eqref{COR1} we set
$$
\widetilde{G}(x)=\frac{2G(x)}{2-\delta}-\frac{\delta}{2}F(x),
$$
or
$$
G(x)=\left(1-\frac{\delta}{2}\right)\widetilde{G}(x)+\frac{\delta}{2}F(x).
$$
Apparently, $\widetilde{G}(x)\in\mathcal{G}(m_1^\prime,
m_2^\prime)$, and, assuming that $0<F(x)<1$ for all $x>0$ according
to the triangle inequality we obtain:
$$
|\widetilde{G}(x)-G(x)|\leq\frac{\delta}{2}\widetilde{G}(x)+\frac{\delta}{2}F(x)<\delta.
$$
The lemma is proved.
\end{proof}

With the aid of Lemmas \ref{lem1} and \ref{cor1} we will solve the following problem. Let $G_1(x)$ and $G_2(x)$ be two probability distributions belonging to the class $\mathcal{G}(\mathcal{M})$. Under the assumption that $\sup_{x>0}|G_1(x)-G_2(x)|<\epsilon$ we will find an estimate for the supremum of $|\gamma_{G_1}-\gamma_{G_2}|$ (the supremum between the corresponding least positive roots of the functional equations $z=\widehat{G}_1(\mu-\mu z)$ and $z=\widehat{G}_2(\mu-\mu z)$.) Solution of this problem, in particular, addresses the case when the probability distribution functions $G_1(x)$ and $G_2(x)$ belong to the class $\mathcal{G}(\frak{g}_1, \frak{g}_2)$. In our analysis below the estimate of \eqref{2.9'} is used.

The analysis uses Lemma \ref{lem1}. The
Laplace-Stieltjes transform
$\widehat{G}_1(s)=\mathrm{e}^{-\frak{g}_1s}$ contains only the parameter $\frak{g}_1$ and does not contain the second one $\frak{g}_2$. Hence similarly to \eqref{2.1} one can write
\begin{equation}\label{2.1''}
\inf_{G\in\mathcal{G}(\frak{g}_1,
m_2)}\int_0^\infty\mathrm{e}^{-sx}\mathrm{d}G(x)=\inf_{G\in\mathcal{G}(\mathcal{M})}
\int_0^\infty\mathrm{e}^{-sx}\mathrm{d}G(x)=\mathrm{e}^{-s\frak{g}_1},
\ s\geq0,
\end{equation}
where $m_2$ is a fictive parameter, which is assumed to be unknown.
Let us find this
unknown parameter $m_2$ in the Laplace-Stieltjes transform
$\widehat{G}_2(s)=1-\frac{\frak{g}_1^2}{m_2}+
\frac{\frak{g}_1^2}{m_2}\exp\left(-\frac{m_2}{\frak{g}_1}s\right)$
taking into account that (cf. \eqref{2.9'})
\begin{equation}\label{2.9''}
\sup_{\substack{G_2\in\mathcal{G}(\frak{g}_1, m_2)\\
\mathcal{K}(G_1,G_2)\leq\epsilon}}\sup_{s>0}\big|\widehat{G}_1(s)-\widehat{G}_2(s)\big|=\epsilon_1.
\end{equation}
Relation \eqref{2.9''} holds true, because
$G_1(x)\in\mathcal{G}(\mathcal{M})$, and according to Lemma
\ref{cor1} for any $\epsilon>0$ there exists a probability
distribution function
$\widetilde{G}_1(x)\in\mathcal{G}(\mathcal{M})$ such that
$|G_1(x)-\widetilde{G}_1(x)|<\epsilon$. On the other hand, the
class of probability distribution functions
$\mathcal{G}(\frak{g}_1, m_2)$ is dense, so $\widetilde{G}_1(x)$
can be chosen  belonging to the same class $\mathcal{G}(\frak{g}_1, m_2)$ as the probability distribution function $G_1(x)$.

The real distance between the Laplace-Stieltjes transforms
$\widehat{G}_1(s)$ and $\widehat{G}_2(s)$
($G_1\in\mathcal{G}(\frak{g}_1,m_2),
G_2\in\mathcal{G}(\frak{g}_1,m_2))$ is
\begin{equation}\label{2.6'''}
\sup_{G_1,G_2\in\mathcal{G}(\frak{g}_1,
m_2)}\sup_{s\geq0}\big|\widehat{G}_1(s)-\widehat{G}_2(s)\big|=1-\frac{\frak{g}_1^2}{m_2}.
\end{equation}
(cf. relation \eqref{2.6}). Therefore, equating the right hand side of \eqref{2.6'''} to $\epsilon_1$ we have
$$
1-\frac{\frak{g}_1^2}{m_2}=\epsilon_1,
$$
and hence
\begin{equation}\label{2.24}
m_2=\frac{\frak{g}_1^2}{1-\epsilon_1}.
\end{equation}
The meaning of the parameter $m_2$ given by \eqref{2.24} is as
follows. If the distance between two Laplace-Stieltjes transforms is
$\epsilon_1$ in the sense of relation \eqref{2.6'''}, then it
remains the same for all distributions $G_1(x)$ and $G_2(x)$
belonging to the family $\mathcal{G}(\frak{g}_1,g_2)$, $m_2\leq
g_2\leq\frak{g}_2$, where the class $\mathcal{G}(\frak{g}_1,m_2)$ is
a marginal class of this family. In this case \eqref{2.9''} can be
simplified as
\begin{equation}\label{2.9'''}
\sup_{\substack{G_2\in\mathcal{G}(\frak{g}_1, m_2)\\
\mathcal{K}(G_1,G_2)\leq\epsilon}}\sup_{s>0}\big|\widehat{G}_1(s)-\widehat{G}_2(s)\big|
=\sup_{G_2\in\mathcal{G}(\frak{g}_1, m_2)}\sup_{s>0}\big|\widehat{G}_1(s)-\widehat{G}_2(s)\big|=\epsilon_1.
\end{equation}

Hence, in this case we have the bounds coinciding with the class of
all distributions of positive random variables having the moments
$m_1=\frak{g}_1$ and $m_2=\frac{\frak{g}_1^2}{1-\epsilon_1}$, i.e.
with the class
$\mathcal{G}\left(\frak{g}_1,\frac{\frak{g}_1^2}{1-\epsilon_1}\right)$.
We also have as follows:
\begin{equation}
\label{2.25}
\begin{aligned}
\sup_{G_1,G_2\in\mathcal{G}\left(\frak{g}_1,\frac{\frak{g}_1^2}{1-\epsilon_1}\right)}
\big|\gamma_{G_1}-\gamma_{G_2}\big|&=1+\frac{\frak{g}_1^2}
{\frac{\frak{g}_1^2}{1-\epsilon_1}}(\ell-1)-\ell\\
&=1+(1-\epsilon_1)(\ell-1)-\ell\\
&=\epsilon_1-\epsilon_1\ell\\
&\leq \epsilon-\epsilon\ell.
\end{aligned}
\end{equation}

Let us consider another case, where
$m_1=\frak{g}_1-\delta\geq\frak{m}$, $\delta>0$. Let
$\widehat{G}_1(s)=\mathrm{e}^{-(\frak{g}_1-\delta)s}$, and let
$\widehat{G}_2(s)=1-\frac{(\frak{g}_1-\delta)^2}{m_2}+
\frac{(\frak{g}_1-\delta)^2}{m_2}\exp\left(-\frac{m_2}{\frak{g}_1-\delta}s\right)$
with an unknown parameter $m_2$. In this case,
$$\sup_{\substack{G_2\in\mathcal{G}(\frak{g}_1-\delta, m_2)\\
\mathcal{K}(G_1,G_2)\leq\epsilon}}\sup_{s>0}|\widehat{G}_1(s)-\widehat{G}_2(s)|$$
cannot be greater than $\epsilon$ (see relations \eqref{2.9} and
\eqref{2.9'}).

 For example, taking
$m_1=\frak{m}$ we arrive at
$\widehat{G}_1(s)\equiv\widehat{G}_2(s),$ and therefore
$$\sup_{\substack{G_2\in\mathcal{G}(\frak{m}, \frak{m}^2)\\
\mathcal{K}(G_1,G_2)\leq\epsilon}}\sup_{s>0}|\widehat{G}_1(s)-\widehat{G}_2(s)|=0.$$
For an arbitrary choice of $m_1=\frak{g}_1-\delta\geq\frak{m}$, one
have
$$\sup_{\substack{G_2\in\mathcal{G}(\frak{g}_1-\delta, m_2)\\
\mathcal{K}(G_1,G_2)\leq\epsilon}}\sup_{s>0}|\widehat{G}_1(s)-\widehat{G}_2(s)|=\epsilon_2\leq\epsilon.$$
(The exact value of $\epsilon_2$ is not important.) In this case, similarly
to \eqref{2.24}
\begin{equation}
\label{2.26} m_2=\frac{(\frak{g}_1-\delta)^2}{1-\epsilon_2},
\end{equation}
and similarly to \eqref{2.25},
\begin{equation}
\label{2.27}
\begin{aligned}
\sup_{G_1,G_2\in\mathcal{G}\left(\frak{g}_1-\delta,\frac{(\frak{g}_1-\delta)^2}
{1-\epsilon}\right)} \big|\gamma_{G_1}-\gamma_{G_2}\big|
&=\epsilon_2-\epsilon_2\ell^*,\\
\end{aligned}
\end{equation}
where $\ell^*$ is the solution of the equation
$z=\mathrm{e}^{-\mu(\frak{g}_1-\delta)+\mu(\frak{g}_1-\delta)z}$. It
is readily seen that $\ell^*>\ell$. (The presence of positive
$\delta$ yields the value of the root of functional equation greater
compared to the case where $\delta$ is not presented (i.e.
$\delta=0$).)

Keeping in mind that $\ell^*>\ell$ and $\epsilon_2\leq\epsilon$,
from \eqref{2.27} we have:
\begin{equation}
\label{2.28}
\begin{aligned}
\sup_{G_1,G_2\in\mathcal{G}\left(\frak{g}_1-\delta,\frac{(\frak{g}_1-\delta)^2}
{1-\epsilon}\right)} \big|\gamma_{G_1}-\gamma_{G_2}\big|
&\leq\epsilon-\epsilon\ell.
\end{aligned}
\end{equation}

Hence, from relations \eqref{2.25} and \eqref{2.28}
we arrive at the following theorem.

\begin{thm}
\label{thm1} For any probability distribution functions $G_1(x)$ and
$G_2(x)$ belonging to the class $\mathcal{G}(\frak{g}_1,
\frak{g}_2)$ and satisfying condition \eqref{AC} we have as follows.

If $\epsilon<1-\frac{\frak{g}_1^2}{\frak{g}_2}$, then
\begin{equation}\label{T1}
\sup_{\substack{G_1,G_2\in\mathcal{G}(\frak{g}_1,
\frak{g}_2)\\\mathcal{K}(G_1,G_2)\leq\epsilon}}\big|\gamma_{G_1}-\gamma_{G_2}\big|\leq\epsilon-\epsilon\ell.
\end{equation}
Otherwise,
$$
\sup_{\substack{G_1,G_2\in\mathcal{G}(\frak{g}_1,
\frak{g}_2)\\\mathcal{K}(G_1,G_2)\leq\epsilon}}\big|\gamma_{G_1}-\gamma_{G_2}\big|
=1+\frac{\frak{g}_1^2}{\frak{g}_2}(\ell-1)-\ell,
$$
where $\ell$ is the least root of the equation
$$
z=\mathrm{e}^{-\mu\frak{g}_1+\mu\frak{g}_1z}.
$$
\end{thm}

\begin{rem}\label{rem1}
Theorem \ref{thm1} is formulated for the class of distributions
$\mathcal{G}(\frak{g}_1,\frak{g}_2)$. As it was mentioned, the
parameters $\frak{g}_1$ and $\frak{g}_2$ are usually unknown. For
practical applications, the errors for these parameters should be
taken into account. Let us assume that ranges of these parameters
are known. For example, $g_1^{\mathrm{lower}}\leq\frak{g}_1\leq
g_1^{\mathrm{upper}}$ and $g_2^{\mathrm{lower}}\leq\frak{g}_2\leq
g_2^{\mathrm{upper}}$ are assumed to be satisfied with a given confidence probability $P$.
It follows from Theorem \ref{thm1} that
if ${g}_2^{\mathrm{lower}}>(g_1^{\mathrm{upper}})^2$ and
$\epsilon<1-\frac{(g_1^{\mathrm{upper}})^2}{{g}_2^{\mathrm{lower}}}$,
 then the lower bound
$\ell$ (the least root of equation \eqref{LR}) for the least positive root $\gamma_G$, should be
replaced by the smaller value given by the least root of the
equation
\begin{equation*}
z=\mathrm{e}^{-\mu g_1^{\mathrm{upper}}+\mu g_1^{\mathrm{upper}}z}.
\end{equation*}
This new value should replace $\ell$ in \eqref{T1} to be used in real applications.

For a nontrivial class $\mathcal{G}(\frak{g}_1,\frak{g}_2)$, and large enough volume of observations $N$, the above condition ${g}_2^{\mathrm{lower}}>(g_1^{\mathrm{upper}})^2$ is natural.

\end{rem}

\section{Asymptotic bounds for characteristics in large loss queueing systems}\label{sec3}
\subsection{Loss probability in the $GI/M/1/n$ queueing system}\label{sec3.1}
In this section we apply the results of Section \ref{sec2} to large
loss $GI/M/1/n$ queueing systems (the parameter $n$ is assumed to be large). The results of this section are elementary.
However, they serve as a basis for the analysis of the more realistic queueing systems,
which are studied in Sections \ref{sec3.1A} and \ref{sec3.2}. The
bounds for the loss probability obtained for this elementary system are then also used for a
more delicate continuity analysis of the loss probability in
$M/M/1/n$ queueing systems in Section \ref{sec4}.

Recall the known asymptotic result for the loss
probability in the $GI/M/1/n$ queueing system as $n\to\infty$.

 Let $\widehat{A}(s)$ denote the Laplace-Stieltjes transform of
the interarrival time probability distribution function $A(x)$,
let $\mu$ denote the reciprocal of the expected service time, let
$\rho$ denote the load,
$\rho=-\frac{1}{\mu\widehat{A}^\prime(0)}$, which is assumed to be
less than 1, and let $\alpha$ denote the positive least root of the
functional equation $z=\widehat{A}(\mu-\mu z)$. It has been shown in
\cite{Abramov 2002} that, as $n\to\infty$, the loss probability
$P_{\mathrm{loss}}(n)$ is asymptotically represented as follows:
\begin{equation}\label{3.1}
P_{\mathrm{loss}}(n)=\frac{(1-\rho)[1+\mu\widehat{A}^\prime(\mu-\mu\alpha)]\alpha^n}
{1-\rho-\rho[1+\mu\widehat{A}^\prime(\mu-\mu\alpha)]\alpha^n}+o(\alpha^{2n}).
\end{equation}

Notice, that the function $\Psi(z)=\widehat{A}(\mu-\mu z)-z$ is a
convex function in variable $z$. There are two roots $z=\alpha$ and $z=1$
in the interval [0,1], and
$\Psi^\prime(\alpha)=-\mu\widehat{A}^\prime(\mu-\mu \alpha)-1>-1$.
Therefore, according to convexity we have the inequality:
\begin{equation}\label{3.1+}
\Psi^\prime(\alpha)\leq-\frac{\Psi(0)}{\alpha}.
\end{equation}
From \eqref{3.1+} we obtain:
\begin{equation*}\label{3.1++}
1+\mu\widehat{A}^\prime(\mu-\mu\alpha)\geq\frac{\widehat{A}(\mu)}{\alpha},
\end{equation*}
and therefore
\begin{equation}\label{3.1+++}
\frac{\widehat{A}(\mu)}{\alpha}\leq
1+\mu\widehat{A}^\prime(\mu-\mu\alpha)\leq1.
\end{equation}

Assume that $A(x)\in\mathcal{G}(\frak{g}_1, \frak{g}_2)$ is
unknown, but the first two moments $\frak{g}_1$ and
$\frak{g}_2$ are given. In this and the following examples we do not discuss the statistical bounds for these moments such as those considered in Remark \ref{rem1}. So, all our examples are built on the basis of the moments $\frak{g}_1$ and
$\frak{g}_2$ only.

Assume that $A_{\mathrm{emp}}(x)$ is an empirical
probability distribution function of this class, its
Laplace-Stieltjes transform is $\widehat{A}_{\mathrm{emp}}(s)$,
the root of the corresponding functional equation
$z=\widehat{A}_{\mathrm{emp}}(\mu-\mu z)$ is $\alpha^*$, and
according to available information, Kolmogorov's
distance between $A_{\mathrm{emp}}(x)$ and $A(x)$ is
$\mathcal{K}(A, A_{\mathrm{emp}})\leq\epsilon$.

Consider the case
$\epsilon<1-\frac{\frak{g}_1^2}{\frak{g}_2}$ ($\frak{g}_1^2\neq
\frak{g}_2$). Since $A(x)$ is unknown, $\widehat{A}(s)$ will be replaced by $\widehat{A}_{\mathrm{emp}}(s)$ in \eqref{3.1+++}. The numerator of the left-hand side of  \eqref{3.1+++} is replaced by the extremal element $\mathrm{e}^{-\mu\frak{g}_1}$, which is not greater than that original. The corresponding denominator is replaced by $(\alpha^*+\epsilon-\epsilon\ell)$, which is not smaller than that original
$\alpha^*$.
Assume that $\epsilon$ is such small that
$\alpha^*-\epsilon+\epsilon\ell>\ell$ and
$\alpha^*+\epsilon-\epsilon\ell<1+\frac{\frak{g}_1^2}{\frak{g}_2}(\ell-1)$.
Then we have:
\begin{equation}\label{3.1++++}
\frac{\mathrm{e}^{-\mu\frak{g}_1}}{\alpha^*+\epsilon-\epsilon\ell}\leq
1+\mu\widehat{A}_{\mathrm{emp}}^\prime(\mu-\mu\alpha^*)\leq1.
\end{equation}
Note, that the assumption on $\epsilon$ under which \eqref{3.1++++} is satisfied can be written as
\begin{equation}\label{CE}
\epsilon<\min\left\{1-\frac{\frak{g}_1^2}{\frak{g}_2}, \ \frac{\alpha^*-\ell}{1-\ell}, \
\frac{\frak{g}_2(1-\alpha^*)-\frak{g}_1^2(1-\ell)}{\frak{g}_2(1-\ell)}\right\}.
\end{equation}

Using \eqref{3.1++++}, in the case of small $\epsilon$ satisfying \eqref{CE}, according to Theorem \ref{thm1} for $n$ large enough
we have the following two inequalities for lower
$\underline{P}(n)$ and upper $\overline{P}(n)$ levels of the loss
probability:
\begin{eqnarray}
\underline{P}(n)&=&\frac{(1-\rho)\mathrm{e}^{-\mu\frak{g}_1}(\alpha^*-\epsilon+\epsilon\ell)^n}
{(1-\rho)(\alpha^*+\epsilon-\epsilon\ell)-\rho\mathrm{e}^{-\mu\frak{g}_1}(\alpha^*-\epsilon+\epsilon\ell)^n},\label{3.2}\\
\overline{P}(n)&=&\frac{(1-\rho)(\alpha^*+\epsilon-\epsilon\ell)^n}
{1-\rho-\rho(\alpha^*+\epsilon-\epsilon\ell)^n}.\label{3.3}
\end{eqnarray}

Therefore, for large $n$ we have the following asymptotic bounds for $
P_{\mathrm{loss}}(n)$:
\begin{equation}\label{3.3+}
\begin{aligned}
&\frac{(1-\rho)\mathrm{e}^{-\mu\frak{g}_1}(\alpha^*-\epsilon+\epsilon\ell)^n}
{(1-\rho)(\alpha^*+\epsilon-\epsilon\ell)-\rho\mathrm{e}^{-\mu\frak{g}_1}(\alpha^*-\epsilon+\epsilon\ell)^n}\\
&\leq P_{\mathrm{loss}}(n)\\ &\leq
\frac{(1-\rho)(\alpha^*+\epsilon-\epsilon\ell)^n}
{1-\rho-\rho(\alpha^*+\epsilon-\epsilon\ell)^n}.
\end{aligned}
\end{equation}

If $\epsilon\geq1-\frac{\frak{g}_1^2}{\frak{g}_2}$, then the terms
$(\alpha^*+\epsilon-\epsilon\ell)$ in \eqref{3.2}, \eqref{3.3} and
\eqref{3.3+} should be replaced by these
$\left[1+\frac{\frak{g}_1^2}{\frak{g}_2}(\ell-1)\right]$, and the
terms $(\alpha^*-\epsilon+\epsilon\ell)$ in \eqref{3.2} and
\eqref{3.3+} should be replaced by $\ell$.

\subsection{Losses from the $M/GI/1$ buffer model}\label{sec3.1A} In this section we obtain lower and upper
bounds for the loss probability of the following $M/GI/1$ buffer
model \cite{Abramov 2004}. Assume that messages (units) arrive in
the buffer of large capacity $N$. Units arrive by batches, the sizes
of which are independent and identically distributed positive
integer random variables $\nu_i$ with expectation $c$. In addition,
the random variables $\nu_i$ are assumed to be bounded, i.e.
$\mathrm{P}\{\nu^{\mathrm{lower}}\leq\nu_i\leq\nu^{\mathrm{upper}}\}=1.$
Interarrival times of batches are exponentially distributed with
parameter $\lambda$, and the service (processing) times of these
batches are independent and identically distributed random variables
with the probability distribution function $B(x)$ and expectation
$b$. If upon arrival of a batch the number of units in the system
exceeds the buffer capacity, then the entire batch loses from the
system. In addition, there is probability $p$ that an arrival batch
of units does not join the system due to transmission error. In all
other situations an arrival batch of units joins the system and
waits for its processing.

Assume that $\rho=\lambda b>1$. Then, for the lost probability the
following representation has been derived in \cite{Abramov 2004}
(see relation (5.3) on page 757):

\begin{equation}\label{MG1.1}
\pi_N=\frac{p+\rho-1}{\rho}\cdot\frac{(\rho-1)+p[1+\lambda
\widehat{B}(\lambda-\lambda\beta)]\mathrm{E}\beta^{\zeta(N)}}{(\rho-1)+[1+\lambda
\widehat{B}(\lambda-\lambda\beta)]\mathrm{E}\beta^{\zeta(N)}}+o(\mathrm{E}\beta^{\zeta(N)}),
\end{equation}
where $\widehat{B}(s)$ denotes the Laplace-Stieltjes transform of the probability distribution function $B(x)$, $\beta$ is the least in absolute value root of the functional equation $z=\widehat{B}(\lambda-\lambda z)$, and
$$
\zeta(N)=\sup\left\{m: \sum_{i=1}^m\nu_i\leq N\right\}.
$$

Notice that since
$\mathrm{P}\{\nu^{\mathrm{lower}}\leq\nu_i\leq\nu^{\mathrm{upper}}\}=1$,
then the similar property for the random variable $\zeta(N)$ is
satisfied:
$\mathrm{P}\{\zeta^{\mathrm{lower}}(N)\leq\zeta(N)\leq\zeta^{\mathrm{upper}}(N)\}=1$,
and, in addition, since as $N\to\infty$
$$
\mathrm{P}\left\{\lim_{N\to\infty}\frac{\zeta(N)}{N}=\frac{1}{c}\right\}=1,
$$
then, as $N\to\infty$,
\begin{equation}\label{MG1.2}
\mathrm{E}\beta^{\zeta(N)}=\beta^{\frac{N}{c}}[1+o(1)].
\end{equation}
Substituting \eqref{MG1.2} into \eqref{MG1.1} we obtain:
\begin{equation}\label{MG1.3}
\pi_N=\frac{p+\rho-1}{\rho}\cdot\frac{(\rho-1)+p[1+\lambda
\widehat{B}(\lambda-\lambda\beta)]\beta^{\frac{N}{c}}}{(\rho-1)+[1+\lambda
\widehat{B}(\lambda-\lambda\beta)]\beta^{\frac{N}{c}}}+o(\beta^{\frac{N}{c}}).
\end{equation}

Similarly to \eqref{3.1+++} for the term
$1+\lambda\widehat{B}^\prime(\lambda-\lambda\beta)$ we have the
inequalities:
\begin{equation}\label{3.7+}
\frac{\widehat{B}(\lambda)}{\beta}\leq1+\lambda\widehat{B}^\prime(\lambda-\lambda\beta)\leq1.
\end{equation}
Assume now that $B(x)\in\mathcal{G}_2(\frak{g}_1, \frak{g}_2)$ is
unknown, but with the first two moments $\frak{g}_1$ and
$\frak{g}_2$ are given, assume that $B_{\mathrm{emp}}(x)$ is an
empirical probability distribution function of this class, its
Laplace-Stieltjes transform is $\widehat{B}_{\mathrm{emp}}(s)$, the
least positive root of the corresponding functional equation
$z=\widehat{B}_{\mathrm{emp}}(\lambda-\lambda z)$ is $\beta^*$, and
assume that according to an available information the Kolmogorov
distance between $B_{\mathrm{emp}}(x)$ and $B(x)$ is $\mathcal{K}(B,
B_{\mathrm{emp}})\leq\epsilon$. Similarly to \eqref{CE} assume that $\epsilon$ satisfy the inequality
\begin{equation}\label{CE1}
\epsilon<\min\left\{1-\frac{\frak{g}_1^2}{\frak{g}_2}, \ \frac{\beta^*-\ell}{1-\ell}, \
\frac{\frak{g}_2(1-\beta^*)-\frak{g}_1^2(1-\ell)}{\frak{g}_2(1-\ell)}\right\}.
\end{equation}

Then similarly to
\eqref{3.1++++} we have
\begin{equation}\label{3.7++}
\frac{\mathrm{e}^{-\lambda\frak{g}_1}}{\beta^*+\epsilon-\epsilon\ell}\leq
1+\lambda\widehat{B}_{\mathrm{emp}}^\prime(\lambda-\lambda\beta^*)\leq1.
\end{equation}
Under same assumption \eqref{CE1},
using \eqref{3.7++} for large $N$ we arrive at the inequalities for
lower, $\underline{\pi}_N$, and upper, $\overline{\pi}_N$, levels of
this loss probability:
$$
\underline{\pi}_N\geq\frac{p+\rho-1}{\rho}\cdot\frac{(\rho-1)(\beta^*+\epsilon-\epsilon\ell)
+p\mathrm{e}^{-\lambda\frak{g}_1}(\beta^*-\epsilon+\epsilon\ell)^{\frac{N}{c}}}
{(\rho-1)(\beta^*+\epsilon-\epsilon\ell)+(\beta^*+\epsilon-\epsilon\ell)^{\frac{N}{c}+1}},
$$
$$
\overline{\pi}_N\leq\frac{p+\rho-1}{\rho}\cdot\frac{(\rho-1)(\beta^*+\epsilon-\epsilon\ell)
+p(\beta^*+\epsilon-\epsilon\ell)^{\frac{N}{c}+1}}
{(\rho-1)(\beta^*+\epsilon-\epsilon\ell)+
\mathrm{e}^{-\lambda\frak{g}_1}(\beta^*-\epsilon+\epsilon\ell)^{\frac{N}{c}}}.
$$
Hence, we arrive at the following statement.

\begin{prop}
Under the above assumptions given in this section for the $M/GI/1$
buffer model with large parameter $N$, in the case
$$
\epsilon<\min\left\{1-\frac{\frak{g}_1^2}{\frak{g}_2}, \ \frac{\beta^*-\ell}{1-\ell}, \
\frac{\frak{g}_2(1-\beta^*)-\frak{g}_1^2(1-\ell)}{\frak{g}_2(1-\ell)}\right\}
$$
for the loss probabilities
$\pi_N$
we have the inequalities:
\begin{equation*}
\begin{aligned}
&\frac{p+\rho-1}{\rho}\cdot\frac{(\rho-1)(\beta^*+\epsilon-\epsilon\ell)
+p\mathrm{e}^{-\lambda\frak{g}_1}(\beta^*-\epsilon+\epsilon\ell)^{\frac{N}{c}}}
{(\rho-1)(\beta^*+\epsilon-\epsilon\ell)+(\beta^*+\epsilon-\epsilon\ell)^{\frac{N}{c}+1}}\\
&\leq\pi_N\\
&\leq\frac{p+\rho-1}{\rho}\cdot\frac{(\rho-1)(\beta^*+\epsilon-\epsilon\ell)
+p(\beta^*+\epsilon-\epsilon\ell)^{\frac{N}{c}+1}}
{(\rho-1)(\beta^*+\epsilon-\epsilon\ell)+
\mathrm{e}^{-\lambda\frak{g}_1}(\beta^*-\epsilon+\epsilon\ell)^{\frac{N}{c}}}.
\end{aligned}
\end{equation*}

\end{prop}

\subsection{The buffers system with priorities}\label{sec3.2} In this section we
study the following special model considered in \cite{Abramov 2008b}
(see also \cite{Abramov 2008c}). This model describes processing
messages in priority queueing systems with large buffers, and the
effective bandwidth problem. Another interpretation of this system
is a transportation system, in which vehicles pass by with some time
intervals to pick up $C$ passengers, at the most, with accordance of
their priority status (different types of passenger are supposed to be). The formal description of the problem, given
in terms of the buffers system with priorities, is as follows.

Suppose that arrival process of customers in the system is a renewal
process $\mathcal{A}(t)$, with the expected value of a renewal
period $\frac{1}{\lambda}$. There are $l$ types of customers, and
there is the probability $p_j>0$ that an arriving customer belongs
to the type $j$ $\bigg(\sum\limits_{j=1}^{l}p^{(j)}=1\bigg)$. Therefore, the time
intervals between arrivals of the type $j$ customers are independent
and identically distributed random variables with expectation
$\frac{1}{\lambda p^{(j)}}$.

Assume that for $i<j$, the customers of the type $i$ have higher
priority than the customers of the type $j$, so customers of type 1 are
those of the highest priority and customers of the type $l$ have the
lowest priority. Assume that customers leave the system by groups of
$C$ as follows. If the number of customers in the system is not
greater than $C$, then all (remaining) customers leave the system.
Otherwise, if the number of customers in the system exceeds the
value $C$, then customers leave according to their priority: a
higher priority customer has an advantage to leave earlier. For
example if $C=5$, $l=3$, and immediately before departure moment
there are three customers of type 1, three customer of type 2 and
one customers of type 3 (i.e. seven customers in total), then after
the departure there will only remain one customer of type 2 and one
customer of type 3 in the system. Times between departures are
assumed to be exponentially distributed with parameter $\mu$. Assume
that $\frac{\lambda}{C\mu}<1$.

The buffer capacities for the type $j$ customers is denoted
$N^{(j)}$. Assume that all of the capacities $N^{(j)}$,
$j=1,2,\ldots,l$ are large enough, i.e. they are assumed to increase
to infinity according to the rule that roughly is explained as
follows. For specific numbers
$0<\alpha_1<\alpha_2<\ldots<\alpha_l<1$, the meaning of which is
explained later, it is assumed that for any $j<k$
\begin{equation}\label{AA}
\alpha_j^{N_j}=o\left(\alpha_k^{N_k}\right),
\end{equation}
where $N_j:=\sum\limits_{i=1}^{j}N^{(i)}$, $j=1,2,\ldots,l$, is the
cumulative buffer content of customers of the first $j$ types.

Let $p_k:=\sum\limits_{j=1}^{k}p^{(j)}$ be the probability of arrival of a
customer of one of the first $k$ types ($p_l\equiv1$). Then the
times between arrivals of customers, who are related to one of the
first $k$ types, $k=1,2,\ldots,l$, are independent and identically
distributed with expectation $\frac{1}{\lambda p_k}$.

Since $\frac{\lambda}{C\mu}<1$, then $\rho_k=\frac{\lambda
p_k}{C\mu}<1$ for all $k=1,2,\ldots,l$. Let $A_k(x)$ denote the
probability distribution function of an interarrival time of the
cumulative arrival process generated by customers of the first $k$ types,
and let $\widehat{A}_k(s)$ ($s\geq0$) denote the Laplace-Stieltjes
transform of $A_k(x)$. For the Laplace-Stieltjes transform
$\widehat{A}_k(s)$ we have:
\begin{equation}\label{3.41}
\begin{aligned}
\widehat{A}_k(s)&=\sum_{i=1}^{\infty}p_k(1-p_k)^{i-1}[\widehat{A}_l(s)]^i\\
&=p_k\widehat{A}_l(s)\frac{1}{1-(1-p_k)\widehat{A}_l(s)}.
\end{aligned}
\end{equation}

Let $\alpha_k$ denote the least positive root of the functional
equation
\begin{equation}\label{3.5}
z=\widehat{A}_k(\mu-\mu z^C)
\end{equation}

(There is a unique root of this functional equation in the
interval (0,1), see \cite{Abramov 2008b}.)

Since $\rho_1<\rho_2<\ldots<\rho_l$, then we also have
$\alpha_1<\alpha_2<\ldots<\alpha_l$. It is shown in \cite{Abramov
2008b} and \cite{Abramov 2008c} that under assumptions \eqref{AA},
the loss probability of type $k$ customers is given by the
asymptotic formula
\begin{equation}\label{3.4}
\begin{aligned}
\pi_k&=\frac{(1-\rho_k)[1+C\mu
\widehat{A}_k^\prime(\mu-\mu\alpha_k^C)]\alpha_k^{N_k}}{(1-\rho_k)(1+\alpha_k+\alpha_k^{2}+\ldots+\alpha_k^{C-1})
-\rho_k[1+C\mu
\widehat{A}_k^\prime(\mu-\mu\alpha_k^C)]\alpha_k^{N_k}}\\
&\ \ \ +o\left(\alpha_k^{2N_k}\right)
\end{aligned}
\end{equation}
(The assumption $\alpha_j^{N_j}=o\left(\alpha_k^{N_k}\right)$,
$j<k$, given in \eqref{AA} actually means that the losses of higher
priority customers occur much more rarely compared to the losses of
lower priority customers.)

Our task is to find lower and upper bounds for $\pi_k$. Note that
asymptotic relation \eqref{3.4} is similar to that \eqref{3.1} of
the stationary loss probability in the $GI/M/1/n$ queueing system
with large $n$. The functional equation \eqref{3.5} is a
more general than that considered before in Sections \ref{sec1} and \ref{sec2} (that functional equation is a particular case when $C=1$). For this functional
equation, the lower and upper bounds are as follows. Let
$\frak{g}_1:=\frac{1}{\lambda}$  and $\frak{g}_2$ denote the first
and, respectively, the second moments of the probability distribution
function $A_l(x)$. (In the sequel,
the notation $\frak{g}_1$  is used instead of $\frac{1}{\lambda}$.) Then,
for $k=1,2,\ldots,l$, from the representation of \eqref{3.41} one
can obtain the first and second moments of the probability
distribution function $A_k(x)$:
$$
\int_0^\infty x\mathrm{d}A_k(x)=\frac{\frak{g}_1}{p_k},
$$
and, respectively,
$$
\int_0^\infty
x^2\mathrm{d}A_k(x)=\frac{2(1-p_k)\frak{g}_{1}^{2}+p_k\frak{g}_2}{p_k^{2}}.
$$
Furthermore, let $\ell$ denote the
least positive root of the functional equation
\begin{equation}\label{3.6-}
z=\exp\left(-\frac{\mu\frak{g}_1+\mu\frak{g}_1z^{C}}{p_k}\right).
\end{equation}
(We use the same notation $\ell$ as it was used for the root of the simpler functional equation in Sections \ref{sec1} and \ref{sec2}, because the consideration of the more general functional equation \eqref{3.6-} leads to
 an elementary extension of the result of Rolski
\cite{Rolski 1972} and consequently to elementary extension of the results in Sections \ref{sec1} and \ref{sec2}). Following this, we have:
\begin{equation}\label{3.6}
\inf \alpha_k=\inf_{A_k\in\mathcal{G}\Big(\frac{\frak{g}_1}{p_k},
\frac{2(1-p_k)\frak{g}_1^2+p_k\frak{g}_2}{p_k^{2}}\Big)}\alpha_{A_k}=\ell,
\end{equation}
and
\begin{equation}\label{3.7}
\sup \alpha_k=\sup_{A_k\in\mathcal{G}\Big(\frac{\frak{g}_1}{p_k},
\frac{2(1-p_k)\frak{g}_1^2+p_k\frak{g}_2}{p_k^{2}}\Big)}\alpha_{A_k}=1+\frac{\frak{g}_1^{2}}
{2(1-p_k)\frak{g}_1^{2}+p_k\frak{g}_2}(\ell-1),
\end{equation}
where $\alpha_{A_k}$ is the notation for the root of the above functional equation associated with the probability distribution $A_k(x)$. (Along with the earlier notation $\alpha_k$, this notation is required for our purposes because it is spoken about the upper and lower bounds associated with the class of probability distribution functions defined in \eqref{3.6}, \eqref{3.7} and the equations appearing later in this section.)

Let us assume now that
$\epsilon<1-\frac{\frak{g}_1^{2}}{2(1-p_k)\frak{g}_1^{2}+p_k\frak{g}_2}$.
Then according to the modified version of Theorem \ref{thm1}
related to this case we have the following:
\begin{equation}\label{3.8}
\sup_{\substack{A_k^\prime,A_k^{\prime\prime}\in\mathcal{G}\Big(\frac{\frak{g}_1}{p_k},
\frac{2(1-p_k)\frak{g}_1^{2}+p_k\frak{g}_2}{p_k^{2}}\Big)\\\mathcal{K}(A_k^\prime,A_k^{\prime\prime})
\leq\epsilon}}\big|\alpha_{A_k^\prime}-\alpha_{A_k^{\prime\prime}}\big|\leq
\epsilon-\epsilon\ell,
\end{equation}
where $\alpha_{A_k^\prime}$ and $\alpha_{A_k^{\prime\prime}}$ are
the versions of $\alpha_k$ corresponding the probability
distribution functions $A_k^\prime(x)$ and $A_k^{\prime\prime}(x)$
of the class $\mathcal{G}\left(\frac{\frak{g}_1}{p_k},
\frac{2(1-p_k)\frak{g}_1^{2}+p_k\frak{g}_2}{p_k^{2}}\right)$.

Similarly to inequality \eqref{3.1+++}, we have:
\begin{equation}\label{3.9}
\frac{\widehat{A}_k(\mu)}{\alpha_k}\leq1+C\mu\widehat{A}_k^\prime(\mu-\mu\alpha_k^C)\leq
1,
\end{equation}
where $\widehat{A}_k^\prime(\cdot)$ in \eqref{3.9} denotes the
derivative of $\widehat{A}_k(\cdot)$.

Assume now that $A_k(x)\in
\mathcal{G}\left(\frac{\frak{g}_1}{p_k},\frac{2(1-p_k)\frak{g}_1^{2}+p_k\frak{g}_2}{p_k^{2}}\right)$
is unknown, but the first two moments
$\frac{\frak{g}_1}{p_k}$ and
$\frac{2(1-p_k)\frak{g}_1^{2}+p_k\frak{g}_2}{p_k^{2}}$ are given. Assume
that $A_{\mathrm{emp},k}(x)$ is the empirical probability
distribution function corresponding the theoretical probability
distribution function $A_k(x)$, and the Laplace-Stieltjes
transform of $A_{\mathrm{emp},k}(x)$ is denoted by
$\widehat{A}_{\mathrm{emp},k}(s)$, $s\geq0$. Let $\alpha_k^*$
denote the least positive root of the functional equation
$z=\widehat{A}_{\mathrm{emp},k}(\mu-\mu z^C)$. Assume also that
according to available information, Kolmogorov's distance between
$A_{\mathrm{emp},k}(x)$ and $A_k(x)$ is
$\mathcal{K}(A_{\mathrm{emp},k}, A_k)\leq\epsilon$, where similarly to \eqref{CE} $\epsilon$
is assumed to satisfy the inequality
\begin{equation*}\label{CE2}
\begin{aligned}
\epsilon<\min \left\{1-\frac{\frak{g}_1^2}{2(1-p_k)\frak{g}_1^2+p_k\frak{g}_2}, \ \frac{\alpha_k^*-\ell}{1-\ell},\frac{[2(1-p_k)\frak{g}_1^2+p_k\frak{g}_2](1-\alpha_k^*)-\frak{g}_1^2(1-\ell)}
{[2(1-p_k)\frak{g}_1^2+p_k\frak{g}_2](1-\ell)}\right\}.
\end{aligned}
\end{equation*}

Then similarly to \eqref{3.1++++} we have
\begin{equation}\label{3.10}
\frac{\exp\left(-\frac{\mu\frak{g}_1}{p_k}\right)}{\alpha_k^*+\epsilon-\epsilon\ell}\leq
1+C\mu\widehat{A}_{\mathrm{emp},k}^\prime(\mu-\mu(\alpha_k^*)^C)\leq1.
\end{equation}
Therefore, taking into account \eqref{3.9} and \eqref{3.10} for
sufficiently large $N_k$ we arrive at the following lower (denoted
by $\underline{\pi}_k(N_k)$) and upper (denoted by
$\overline{\pi}_k(N_k)$) values for probability $\pi_k$:
$$
\underline{\pi}_k(N_k)=\frac{(1-\rho_k)\exp\left(-\frac{\mu\frak{g}_1}{p_k}\right)(\alpha_k^*-\epsilon+\epsilon\ell)^{N_k}}
{(1-\rho_k)\sum\limits_{i=0}^C(\alpha_k^*+\epsilon-\epsilon\ell)^i-\rho_k\exp\left(-\frac{\mu\frak{g}_1}{p_k}\right)(\alpha_k^*
-\epsilon+\epsilon\ell)^{N_k}},
$$
$$
\overline{\pi}_k(N_k)=\frac{(1-\rho_k)(\alpha_k^*+\epsilon-\epsilon\ell)^{N_k}}
{(1-\rho_k)\sum\limits_{i=0}^C(\alpha_k^*-\epsilon+\epsilon\ell)^i-\rho_k(\alpha_k^*
+\epsilon-\epsilon\ell)^{N_k}}.
$$
Hence, we arrive at the following statement.

\begin{prop}Under the above assumptions given in this section, in the case where
\begin{equation*}
\begin{aligned}
\epsilon<\min \left\{1-\frac{\frak{g}_1^2}{2(1-p_k)\frak{g}_1^2+p_k\frak{g}_2}, \ \frac{\alpha_k^*-\ell}{1-\ell},
\frac{[2(1-p_k)\frak{g}_1^2+p_k\frak{g}_2](1-\alpha_k^*)-\frak{g}_1^2(1-\ell)}
{[2(1-p_k)\frak{g}_1^2+p_k\frak{g}_2](1-\ell)}\right\}.
\end{aligned}
\end{equation*}
for the loss probabilities $\pi_k$, $k=1,2,\ldots,l$, we have:
\begin{equation*}
\begin{aligned}
&\frac{(1-\rho_k)\exp\left(-\frac{\mu\frak{g}_1}{p_k}\right)(\alpha_k^*-\epsilon+\epsilon\ell)^{N_k}}
{(1-\rho_k)\sum\limits_{i=0}^C(\alpha_k^*+\epsilon-\epsilon\ell)^i-\rho_k\exp\left(-\frac{\mu\frak{g}_1}{p_k}\right)(\alpha_k^*
-\epsilon+\epsilon\ell)^{N_k}}\\
&\leq\pi_k
\leq\frac{(1-\rho_k)(\alpha_k^*+\epsilon-\epsilon\ell)^{N_k}}
{(1-\rho_k)\sum\limits_{i=0}^C(\alpha_k^*-\epsilon+\epsilon\ell)^i-\rho_k(\alpha_k^*
+\epsilon-\epsilon\ell)^{N_k}}.
\end{aligned}
\end{equation*}
\end{prop}

\section{Continuity of the loss probability in the $M/M/1/n$ queueing
 system}\label{sec4} The results of Section \ref{sec2} enable us to
establish continuity of the $M/M/1/n$ queueing system when $n$ is
large. The continuity of the $M/M/1/n$ queueing system was studied
in \cite{Abramov 2008}. In contrast to \cite{Abramov 2008} where by
continuity of $M/M/1/n$ queueing system it is meant the continuity
of a $M/GI/1/n$ queueing system, which is close to the $M/M/1/n$
queueing system, in the present paper by continuity of the $M/M/1/n$
queueing system it is meant the continuity of a $GI/M/1/n$ queueing
system, which is close to that $M/M/1/n$ queueing system. Then, in
the case when parameter $n$ is large, the analysis becomes much
simpler compared to the case when $n$ is not assumed to be large.
(In \cite{Abramov 2008} Conditions (A) and (B) mentioned below are
applied to the probability distribution function of a service time.)

Our assumptions here are similar to those of \cite{Abramov 2008}.
Let $A(x)$ denote probability distribution function of
interarrival time, which slightly differs from the exponential
distribution $E_\lambda(x)=1-\mathrm{e}^{-\lambda x}$ as indicated
in the cases below.

\smallskip

$\bullet$ Condition (A). The probability distribution function
$A(x)$ has the representation
\begin{equation}\label{4.1}
A(x)=pF(x)+(1-p)E_\lambda(x), \ 0<p\leq1,
\end{equation}
where $F(x)=\mathrm{Pr}\{\zeta\leq x\}$ is a probability
distribution function of a nonnegative random variable having the
expectation $\frac{1}{\lambda}$, and
\begin{equation}\label{4.2}
\sup_{x,y\geq0}\left|F_y(x)-F(x)\right|<\epsilon, \ \epsilon>0,
\end{equation}
where $F_y(x)=\mathrm{Pr}\{\zeta\leq x+y|\zeta>y\}$. Relation
\eqref{4.2} says that the distance in Kolmogorov's metric
between $F(x)$ and $E_\lambda(x)$, according to the
characterization theorem of Azlarov and Volodin \cite{Azlarov and
Volodin 1986} (see also \cite{Abramov 2008}), is not greater than
$2\epsilon$.

\smallskip
$\bullet$ Condition (B). Along with \eqref{4.1} and \eqref{4.2} it
is assumed that $F(x)$ belongs either to the class NBU or to the
class NWU.

Recall that a probability distribution function $\Xi(x)$ of a
nonnegative random variable is said to belong to the class NBU if
for all $x\geq0$ and $y\geq0$ we have
$\overline{\Xi}(x+y)\leq\overline{\Xi}(x)\overline{\Xi}(y)$, where
$\overline{\Xi}(x)=1-{\Xi}(x)$. If the opposite inequality holds,
i.e. $\overline{\Xi}(x+y)\geq\overline{\Xi}(x)\overline{\Xi}(y)$,
then $\Xi(x)$ is said to belong to the class NWU.

\smallskip
Under both of these Conditions (A) and (B) we assume that
$\mathrm{E}\zeta^{2}<\infty$ is given.

\medskip
Under Condition (A), we have
\begin{equation}\label{4.4}
\begin{aligned}
\sup_{x>0}\left|A(x)-E_\lambda(x)\right|&=\sup_{x>0}\left|pF(x)-(1-p)E_\lambda(x)-E_\lambda(x)\right|\\
&=p\sup_{x>0}\left|F(x)-E_\lambda(x)\right|.
\end{aligned}
\end{equation}
According to the aforementioned characterization theorem of
Azlarov and Volodin,
$$
\sup_{x>0}\left|F(x)-E_\lambda(x)\right|<2\epsilon.
$$
Therefore, from \eqref{4.4} we obtain
\begin{equation}\label{4.5}
\sup_{x>0}\left|A(x)-E_\lambda(x)\right|<2p\epsilon.
\end{equation}
We also have:
$$
\int_0^\infty
x^{2}\mathrm{d}A(x)=p\mathrm{E}\zeta^{2}+\frac{2(1-p)}{\lambda^{2}}.
$$
Apparently,
$\mathrm{E}\zeta^{2}\geq(\mathrm{E}\zeta)^{2}=\frac{1}{\lambda^{2}}$.
Denote $\mathrm{E}\zeta^{2}=\sigma^{2}+\frac{1}{\lambda^{2}}$,
assuming that $\sigma^2>\frac{1}{\lambda^{2}}$. Thus, it is assumed that $\mathrm{E}\zeta^{2}>\frac{2}{\lambda^{2}}$.


Now one can apply the estimate given by \eqref{T1} in Theorem \ref{thm1}, to obtain
continuity bounds for the loss probability in the case of large
$n$. In this estimate, $\ell$ is the least positive root of the
equation
$z=\exp\left(-\frac{\mu}{\lambda}+\frac{\mu}{\lambda}z\right)$. It is not difficult to check that
the least positive root of this functional equation is $\rho=\frac{\lambda}{\mu}$, and, because of the assumption $\sigma^2>\frac{1}{\lambda^2}$, the value $\rho$ is within the bounds:
\begin{equation}\label{c1}
\ell\leq\rho\leq1+\frac{\frac{1}{\lambda^{2}}}{\sigma^{2}+\frac{1}{\lambda^{2}}}(\ell-1).
\end{equation}
%
%
%

So, keeping in mind the relation \eqref{3.3+} for the loss probability, where in the given case $\alpha^*$ is replaced by $\rho$, we arrive at the following statement.

\begin{prop}
Under Condition (A) and under the assumption $\sigma^2>\frac{1}{\lambda^{2}}$
the following inequalities for the loss probability, as $n\to\infty$, hold:
\begin{equation*}\label{4.3+}
\begin{aligned}
&\frac{(1-\rho)\mathrm{e}^{-\frac{1}{\rho}}(\rho-2p\epsilon_1+2p\epsilon_1\ell)^n}
{(1-\rho)(\rho+2p\epsilon_1-2p\epsilon_1\ell)-\rho\mathrm{e}^{-\frac{1}{\rho}}(\rho-2p\epsilon_1+2p\epsilon_1\ell)^n}\\
&\leq P_{\mathrm{loss}}(n)\\ &\leq
\frac{(1-\rho)(\rho+2p\epsilon_2-2p\epsilon_2\ell)^n}
{1-\rho-\rho(\rho+2p\epsilon_2-2p\epsilon_2\ell)^n},
\end{aligned}
\end{equation*}
where 
$$
\epsilon_1=\min\left\{\rho-\ell, 2p\epsilon(1-\ell)\right\},
$$
and
$$
\epsilon_2=\min\left\{1+\frac{1}{1+\lambda^{2}\sigma^{2}}(\ell-1)-\rho, 2p\epsilon(1-\ell)\right\}.
$$
\end{prop}

Under Condition (B) we have \eqref{4.4}, where under the
additional assumption that $F(x)$ belongs either to the class NBU
or to the class NWU one should apply Lemma 3.1 of \cite{Abramov
2008} rather then the characterization theorem of Azlarov and
Volodin \cite{Azlarov and Volodin 1986}, \cite{Abramov 2008}. In
this case we have
$$
\sup_{x>0}\left|F(x)-E_\lambda(x)\right|<\epsilon.
$$
Therefore, from \eqref{4.4} we obtain
\begin{equation*}\label{4.5+}
\sup_{x>0}\left|A(x)-E_\lambda(x)\right|<p\epsilon.
\end{equation*}

In this case we have the following statement.

\begin{prop}
Under Condition (B) and under the assumption $\sigma^2>\frac{1}{\lambda^{2}}$
the following inequalities for the loss probability, as $n\to\infty$, hold:
\begin{equation*}\label{4.4+}
\begin{aligned}
&\frac{(1-\rho)\mathrm{e}^{-\frac{1}{\rho}}(\rho-p\epsilon_3+p\epsilon_3\ell)^n}
{(1-\rho)(\rho+p\epsilon_3-p\epsilon_3\ell)-\rho\mathrm{e}^{-\frac{1}{\rho}}(\rho-p\epsilon_3+p\epsilon_3\ell)^n}\\
&\leq P_{\mathrm{loss}}(n)\\ &\leq
\frac{(1-\rho)(\rho+p\epsilon_4-p\epsilon_4\ell)^n}
{1-\rho-\rho(\rho+p\epsilon_4-p\epsilon_4\ell)^n},
\end{aligned}
\end{equation*}
where
$$
\epsilon_3=\min\left\{\rho-\ell, p\epsilon(1-\ell)\right\},
$$
and
$$
\epsilon_4=\min\left\{1+\frac{1}{1+\lambda^{2}\sigma^{2}}(\ell-1)-\rho, p\epsilon(1-\ell)\right\}.
$$
\end{prop}

\section*{Acknowledgement}
The author acknowledges with thanks the support of the Australian
Research Council, grant \#DP0771338. The author thanks Prof. Kais Hamza for a fruitful discussion of the results and bringing to the attention of the author an alternative way of the proof of Lemmas \ref{lem1} and \ref{cor1}.


\end{document}